\documentclass[a4paper,11pt]{article}
\usepackage{amsmath}
\usepackage{amsfonts}
\usepackage{supertabular}
\usepackage[latin9]{inputenc}
\usepackage[english]{varioref}
\usepackage{dcolumn}
\usepackage[height=22cm , width = 16cm , top = 4cm , left = 3cm, a4paper]{geometry}
\usepackage[a4paper]{geometry}
\usepackage[final]{graphicx}
\usepackage{epsfig}
\usepackage{pstricks}
\usepackage{psfrag}
\usepackage{rotating}
\usepackage{supertabular}
\usepackage{booktabs}
\usepackage{delarray}
\usepackage{rotating}
\usepackage{subfigure}
\usepackage{nextpage}
\usepackage{layout}
\usepackage{amsthm}
\usepackage{dsfont}
\usepackage{hyperref}
\usepackage[hypcap]{caption}
\usepackage{color}
\numberwithin{equation}{section}

\theoremstyle{plain}
\newtheorem{thm}{Theorem}[section]

\newtheorem{lem}[thm]{Lemma}

\newtheorem{rmkk}[thm]{Remark}

\newcommand{\enter}{\bigskip}

\date{}
\begin{document}
\author{{Prasanta Kumar Barik and Ankik Kumar Giri}\\
\footnotesize Department of Mathematics, Indian Institute of Technology Roorkee,\\ \small{ Roorkee-247667, Uttarakhand, India}\\
  }
\title{Global classical solutions to the continuous coagulation equation with collisional breakage}

\maketitle


\begin{quote}
{\small {\em\bf Abstract.} Existence and uniqueness of mass-conserving classical solutions to the continuous coagulation equation with collisional breakage are investigated for an unbounded class of collision kernels and a particular case of distribution function. The distribution function may have a possibility to attain singularity at the origin. The proof of the existence result relies on the compactness method. Moreover, a uniqueness result is shown. In addition, it is observed that the uniqueness solution is mass-conserving. \enter
}
\end{quote}

\vskip 11pt 

\section{Introduction}
Coagulation and breakage processes occur in the dynamics of particle growth and describe the time evolution of a system of particles under the combined effect of coagulation and breakage. Such processes have a variety of applications in the different field of science, engineering and technology. For instance, in aerosol science (solid or liquid particles suspended in a gas), in astrophysics (formation of stars and planets), polymer science, haematology (red blood cell aggregation) and population dynamics (animal grouping). In these situations, each particle is usually assumed to be fully identified by their size (or volume) which is a nonnegative real number. In this article, we consider a fully nonlinear model which is \emph{continuous coagulation and collisional breakage equation}, see \cite{Safronov:1972, Wilkins:1982, Laurencot:2001}. This model is a partial integro-differential equation which is expressed by unknown quantity $g=g(z, t)$, the concentration of the particle of volume (mass) $z \in [0, \infty)$ at time $t \in[0, \infty)$ as
  \begin{eqnarray}\label{cecbeclassi}
\frac{\partial g}{\partial t}=\mathcal{C}(g)-\mathcal{B}(g)+\mathcal{B}^{*}(g),
\end{eqnarray}
with the initial data
\begin{eqnarray}\label{4in1}
g(z,0) = g_{0}(z)\geq 0,
\end{eqnarray}
  where
\begin{eqnarray*}
\hspace{-5cm} \mathcal{C}(g)(z, t) := \frac{1}{2}\int_0^z E(z-z_1, z_1 ) \mathcal{G}(z-z_1, z_1, t)dz_1,
\end{eqnarray*}
\begin{eqnarray*}
\hspace{-8cm} \mathcal{B}(g)(z, t) : = \int_{0}^{\infty} \mathcal{G}(z, z_1, t)dz_1,
\end{eqnarray*}
and
 \begin{eqnarray*}
\mathcal{B}^*(g) (z, t) :=  \frac{1}{2} \int_{z}^{\infty} \int_{0}^{z_1}P(z|z_1-z_2;z_2) E_1(z_1-z_2, z_2 )  \mathcal{G}(z_1-z_2, z_2, t)dz_2 dz_1,
\end{eqnarray*}
with
\begin{eqnarray*}
\hspace{-7cm} \mathcal{G}(z, z_1, t) :=\Phi(z,z_1)g(z,t)g(z_1,t).
\end{eqnarray*}
We recall that the mechanisms taken into account in this model are the coalescence of two particles to form a larger one and breakup into smaller pieces via binary collision, and that the collision kernel $ 0 \le \Phi(z,z_1) = \Phi(z_1, z),\ (z, z_1) \in [0, \infty) \times [0, \infty)$ models the likelihood that two particles with respective volumes $z$ and $z_1$ merge into a single one $(z + z_1)$ with efficiency of coalescence $E(z, z_1)$. Here, $E_1(z, z_1)$ is the efficiency that the two collide particles are breaking with possible transfer of volume to form two or more particles. In addition, $E(z, z_1)+E_1(z, z_1)=1$ with $0 \le E(z, z_1)=E(z_1, z), E_1(z, z_1)=E_1(z_1, z) \le 1$. Here, the distribution function $P(z|z_1;z_2)$ describes an average number of particles of volume $z$ emerged from the breakage event arising due to the collision between particles of volumes $z_1-z_2$ and $z_2$ with possible transfer of volume. \\

Moreover, the distribution function $P$ enjoys the following properties:
\begin{itemize}
  \item  $ 0 \le P(z|z_1; z_2)=P(z|z_2; z_1)$ (symmetric with respect to $z_1$ and $z_2$).\\
  \item The total number of particles resulting from the collisional breakage process is given by
\begin{align}\label{4Total particles}
\int_{0}^{z_1+z_2}P(z|z_1; z_2)dz = N,\ \text{for all}\  z_1\ge 0,\ z_2\ge 0, \  P(z|z_1;z_2)=0\  \text{for}\  z> z_1+z_2.
\end{align}
  \item  A necessary condition for the mass conservation during the collisional breakage event is
\begin{align}\label{mass1}
\hspace{-5cm}\int_{0}^{z_1+z_2} z P(z|z_1;z_2)dz = z_1+z_2,\ \ \text{for all}\ \ z_1 \ge 0,\ z_2 \ge 0.
\end{align}
 From the condition (\ref{mass1}), the total volume $z_1+z_2$ of particles is conserved during the breakage of a particles of volumes $z_1$ $\&$ $z_2$ due to their collisions.
\end{itemize}

 Throughout the rest of the paper, we assume the following particular distribution function: $P(z|z_1; z_2)=(\theta+2) \frac{z^{\theta}}{(z_1+z_2)^{\theta+1}}$  with $\theta=0$. Hence, we obtain
\begin{align}\label{binarybreakage}
P(z|z_1; z_2)=\frac{2}{(z_1+z_2)}.
 \end{align}
After substituting this distribution function into \eqref{4Total particles}, one can easily obtain the case of binary breakage of particles, i.e. $N=2$.\\

The term $\mathcal{C}(g)$ on the right-hand side of \eqref{cecbeclassi} represents the formation of particles of volumes $z$ due to the occurrence of coagulation events and the second term $\mathcal{B}(g)$ describes the disappearance of particles of volumes $z$ due to both the coagulation and the breakage processes. On the other hand, the term $\mathcal{B}^*(g)$ shows the birth of particles of volumes $z$ due to collision breakage between a pair of particles of volumes $z_1-z_2$ and $z_2$ with possible transfer of masses (volumes).\\

Moreover, throughout this paper, we need to have some information on the time evolution of moments of solutions to \eqref{cecbeclassi}. For a function $g : [0, \infty) \rightarrow \mathds{R} $, we define the following explicit form of moments as
\begin{eqnarray}\label{moment}
\mathcal{M}_r (t) := \int_0^{\infty} z^r g(z, t)dz,\ \ \text{for}\ r \in \mathds{R}.
\end{eqnarray}
 The zeroth $(r =0)$ and first $(r =1)$ moments, $\mathcal{M}_0 (t)$ and $\mathcal{M}_1 (t)$, respectively, represent the total number of particles and the total mass of particles. In coagulation events, the zeroth moment $\mathcal{M}_0 (t)$ is a  decreasing function while in the breakage process, it is an increasing function. However, $\mathcal{M}_1 (t)$ may or may not be constant during the coagulation and collisional breakage processes that depend on the nature of collision rates. Negative moments have been considered in many articles, see \cite{Escobedo:2004, Escobedo:2006, Norris:1999}.\\

Some particular cases of \eqref{cecbeclassi} have been considered by mathematicians and we mention some of them now. Besides the classical coagulation equation which is obtained from \eqref{cecbeclassi} by setting $E(z, z_1)=1$, we can also consider the case where the collision of a $z-$particle and a $z_1-$particle results in either the coalescence of both in a $(z+z_1)$ or in an elastic collision leaving the incoming clusters unchanged.
The system \eqref{cecbeclassi} then reduces to the continuous version of \emph{Smoluchowski coagulation equation} with coagulation coefficients $(E(z, z_1) \Phi(z, z_1))$, see \cite{Barik:2018Mass, Galkin:1986}.\\

The present work mainly deals with the existence and uniqueness of classical solutions to continuous coagulation model with collisional breakage, \eqref{cecbeclassi}--\eqref{4in1}. At the end, it is also observed that the solution satisfies the mass conservation. There are several mathematical results available on the existence and uniqueness of solutions to coagulation and linear breakage equations which are obtained by applying various techniques under different growth conditions on coagulation and breakage kernels, see \cite{Banasiak:2012I, Banasiak:2012II, Banasiak:2011, Barik:2017Anote, Barik:2018Mass, Dubovskii:1996, Giri:2012, Stewart:1989, Stewart:1990}. However, the nonlinear breakage equation has not well known in the mathematics community. In addition, there are a few articles, in which analytical solutions to nonlinear breakage equation have been investigated only for specific collision and breakup kernels, see \cite{Cheng:1988, Cheng:1990,  Kostoglou:2000, Matthieu:2007}. In 1988, the nonlinear breakage model was studied by Cheng and Redner \cite{Cheng:1988}. In \cite{Cheng:1988}, authors have discussed the scaling form of cluster size distribution and asymptotic behaviour of solutions to the continuous nonlinear breakage equation. Moreover, they have shown the basic difference of the scaling solutions to both linear and nonlinear breakage equations by taking some specific homogeneous collision kernel such as $C(az,az_1)=a^{\lambda}C(z, z_1)$ and the homogeneous breakup kernel such as $B(az|az_1;az_2)=a^{-1}B(z|z_1;z_2)$. In 1990, Cheng and  Redner \cite{Cheng:1990} have proposed a specific class of splitting model for the nonlinear breakage equation in which a pair of particles collide with each other. As a result of this collision, both particles are splitting in different ways: $(i)$ in exactly two, $(ii)$ only the large one is splitting or $(iii)$ only the smaller one is splitting. They have also derived asymptotic behaviour of the scaling solution by using homogeneous collision kernel $C(z, z_1)=(zz_1)^{\frac{\lambda}{2}}$ and breakup kernel in different splitting model for nonlinear breakage equation. Later, Kostoglou and Karabelas \cite{Kostoglou:2000}, have discussed an analytical and asymptotic information of solution to the nonlinear breakage equation. They have considered different simple homogeneous collision and breakup kernels to transform the nonlinear breakage equation into linear one for discussing the self-similar solutions. Recently, Ernst and Pagonabarra \cite{Matthieu:2007} have inquired some more details about the scaling solutions and occurrence of shattering transition for different breakage models such as symmetric breakage, L-breakage and S-breakage of nonlinear breakage equation. Here symmetric breakage, L-breakage and S-breakage denote, respectively, the splitting of both particles into exactly two pieces, splitting of the large particle only and the smaller particle only, see \cite{Cheng:1990}. However, it is quite delicate to handle mathematically the continuous nonlinear breakage equation because small sized particles are fragmented into very small-sized to form an infinite number of clusters in a finite time. In order to overcome this problem, we consider a fully nonlinear continuous model known as the \emph{continuous coagulation model with collisional breakage} \eqref{cecbeclassi}. In \cite{Laurencot:2001, Safronov:1972, Wilkins:1982}, authors have discussed the coagulation and collisional breakage equation. In particular, in \cite{Laurencot:2001}, authors have solved the discrete coagulation and collisional breakage equation mathematically by using a weak $L^1$ compactness method. In \cite{Safronov:1972, Wilkins:1982}, the continuous version of coagulation and collisional breakage equation is described. In the present work, we look for global classical solutions to \eqref{cecbeclassi}--\eqref{4in1}. Several researchers have already discussed the existence of global classical solutions to the coagulation and linear breakage equations through various techniques, see \cite{Banasiak:2012I, Banasiak:2012II, Banasiak:2011, Banasiak:2013, Dubovskii:1996, Galkin:1986, Lamb:2004, McLaughlin:1997I, McLaughlin:1997II}. In \cite{Banasiak:2012I, Banasiak:2012II, Banasiak:2011, Banasiak:2013, Lamb:2004, McLaughlin:1997I, McLaughlin:1997II}, authors have discussed the existence of global classical solutions for the coagulation and linear breakage equations by using semigroup technique whereas in \cite{Dubovskii:1996, Galkin:1986}, a different approach, introduced by Galkin and Dubovskii, is used to show the existence of global classical solutions which relies on a compactness argument. In \cite{Galkin:1986}, they have discussed the existence and uniqueness of global solutions to pure coagulation equation by taking an account of unbounded coagulation kernels. Later, in 1996, Dubovskii and Stewart \cite{Dubovskii:1996} have extended this result for the continuous coagulation and linear binary breakage equations with unbounded coagulation and breakage rates. Best to our knowledge, this is the first attempt to address the existence, uniqueness and mass-conservation of classical solutions to \eqref{cecbeclassi}--\eqref{4in1} by using a compactness technique. The main motivation of this work comes from \cite{Galkin:1986, Dubovskii:1996}.\\

The paper is arranged as follows: In Section 2, we state some definitions and hypotheses which are essentially required for upcoming  results in subsequent sections. In Section 3, we show the existence of classical solutions to \eqref{cecbeclassi}--\eqref{4in1} by using a compactness method which has been widely discussed for coagulation equation with linear breakage, see \cite{Galkin:1986, Dubovskii:1996}. The uniqueness result of the existing solutions to \eqref{cecbeclassi}--\eqref{4in1} is studied in the last section. Finally, the mass-conserving property of the solutions is verified  in this section.

\section{Function spaces and Assumptions}
Fix $T>0$ and let us define the following abstract spaces. Let $\Xi$ be the strip defined as
\begin{eqnarray*}
\Xi := \{(z, t): z\in [0, \infty),\   t \in [0, T] \}
\end{eqnarray*}
and $\Xi(0, Z^o; T)$ denotes the following closed rectangle
\begin{eqnarray*}
\Xi(0, Z^o; T) := \{(z, t): z\in [0, Z^o],\ t \in [0, T] \}.
\end{eqnarray*}
 Let $\Lambda_{-\sigma_1, \sigma_2}(T)$ be the space of all continuous functions $g$ with bounded norms defined by
\begin{eqnarray*}
\|g\|_{-\sigma_1, \sigma_2}:=\sup_{0 \leq t \leq T} \int_0^{\infty} (z^{-\sigma_1}+z^{\sigma_2} )|g(z, t)|dz, \  1 < \sigma_2 \le 2 \ \& \ \frac{1}{2} \le \sigma_1 < 1
\end{eqnarray*}
and
\begin{eqnarray*}
\Lambda^{+}_{-\sigma_1, \sigma_2}(T):= \{ g \in \Lambda_{-\sigma_1, \sigma_2 }(T): g \geq 0 \}
\end{eqnarray*}
which is the positive cone of $\Lambda_{-\sigma_1, \sigma_2}(T)$.
\begin{rmkk}
One can see that  $\Lambda_{-\sigma_1, \sigma_2}(T) \subset  \Lambda_{-\sigma_1, 1}(T) \subset  \Lambda_{-\sigma_1, 0}(T)$ \& $\Lambda_{-\sigma_1, \sigma_2}(T) \subset  \Lambda_{0, \sigma_2}(T)$.
\end{rmkk}

In order to show the existence, uniqueness and mass-conserving of classical solutions to \eqref{cecbeclassi}--\eqref{4in1}, we consider the following assumptions on the collision kernel and the distribution function throughout next sections.\\
\textbf{Assumptions}: $(\alpha_1)$ Let $\Phi(z, z_1)$ and $E(z, z_1)$ be non-negative and continuous functions for all $(z, z_1) \in [0, \infty) \times [0, \infty ) $,\\
\\
$(\alpha_2)$ $\Phi(z, z_1)  = c  (z^{\alpha}z_1^{\alpha'}+z^{\alpha'}z_1^{\alpha})$ for all $(z, z_1)\in [0, \infty) \times [0, \infty )$, for a non-negative constant $c$ and $0< \alpha, \alpha' \le 1/2$,\\
\\
$(\alpha_3)$ for all $ r^{\ast} \in [0, 1)$ and there is a positive constant $\zeta(r^{\ast}) >1$ such that
\begin{align*}
\int_{0}^{z_1+z_2} z^{-r^{\ast}} P(z|z_1;z_2)dz \leq \zeta(r^{\ast}) {(z_1+z_2)}^{-r^{\ast}}.
\end{align*}
Note that, for $r^{\ast}=0$ in $(\alpha_3)$, we have $\zeta(r^{\ast}) =2$.
\begin{rmkk}\label{distribution alpha3}
One can see that for each $r > 1$, we have
\begin{align*}
\int_{0}^{z_1+z_2} z^r P(z|z_1; z_2)dz = \frac{2}{(r+1)} {(z_1+z_2)}^r.
\end{align*}
\end{rmkk}
 Now, let us end up this section by stating the following theorem on existence, uniqueness and mass-conservation of solutions to (\ref{cecbeclassi})--(\ref{4in1}).
\begin{thm}\label{ClassicalExistence Thm}
 Assume $(\alpha_1)$--$(\alpha_3)$ and \eqref{binarybreakage} hold. Let the initial value $g_0 \in \Lambda^+_{-\sigma_1, \sigma_2}(0)$. Then \eqref{cecbeclassi}--\eqref{4in1} has a unique mass-conserving solution $g$ in $\Lambda^+_{-\sigma_1, \sigma_2}(T) \cap ( \Lambda^+_{-1/2, -\sigma_1}(T) \cup  \Lambda^+_{1, \sigma_2}(T))$ for additional restrictions on $\sigma_1$ and $\sigma_2$ s.t. $1 < \max \{ 1+\alpha, 1+\alpha' \} \le  \sigma_2 \le 2$ with $ \frac{1}{2} \le \max \{(1-\alpha), (1-\alpha') \}\le \sigma_1 <1$.
\end{thm}

\newpage
\section{Existence of classical solutions}\label{Existenceclassi}

\subsection{Approximated Solutions}
In this section, we first construct a sequence of continuous kernels $(\Phi_n)_{n=1}^{\infty} $ with compact support for each $n \geq 1$, such that
\begin{equation}\label{collkernel}
\Phi_n(z, z_1)= \chi_{(0, n)}(z) \chi_{(0, n)}(z_1) \Phi(z, z_1).
\end{equation}
Using \eqref{collkernel}, we truncate \eqref{cecbeclassi}--\eqref{4in1} as
\begin{align}\label{4trun cccecf}
\frac{\partial g^n(z,t)}{\partial t}  = &\frac{1}{2} \int_{0}^{z}  E(z-z_1, z_1)G^n(z-z_1,z_1, t)dz_1 
- \int_{0}^{n}G^n(z, z_1, t)dz_1\nonumber\\ 
 & +\frac{1}{2} \int_{z}^{n}\int_{0}^{z_1} P(z|z_1-z_2; z_2) E_1(z_1-z_2, z_2) G^n(z_1-z_2, z_2, t)dz_2 dz_1,
\end{align}
 with initial datum
\begin{equation}\label{4trunc ccecfin1}
g^{n}_0(z)=\chi_{(0, n)}(z) g_0(z),
\end{equation}
where $G^n(z,z_1, t):=\Phi_n(z, z_1)g^n(z, t)g^n(z_1, t) $. For each $n\in \mathds{N}$, we may follow as in \cite[Theorem 3.1]{Stewart:1989} or \cite{Walker:2002} by applying classical fixed point theorem to show the approximated equations (\ref{4trun cccecf})--(\ref{4trunc ccecfin1}) has a unique non-negative solution $\tilde{g}^n$. This family of solutions $(\tilde{g}^n)_{n=1}^{\infty}$ belongs to space $ \mathcal{C}'([0, T];L^1(0, n))$ . Additionally, it satisfies (alternative proof is given in the next section)
\begin{align}\label{4trunc mass1}
\int_{0}^{n} z \tilde{g}^n(z, t)dz \leq  \int_{0}^{n} z\tilde{g}^n_0(z)dz\ \ \ \forall t\in (0, T].
\end{align}
 Next, we enlarge the domain of the approximated solution $\tilde{g}^n$ by its zero extension in $[0, \infty)$ as
\begin{equation}\label{4trunc soln}
\tilde{g}^{n}(z, t):=\begin{cases}
g^n(z, t),\ \ & \text{if}\ 0 \le z < n, \\
\text{0},\ \ &  \ \ \ \ \text{otherwise},
\end{cases}
\end{equation}
for $n \in \mathbb{N}$. For conveniently, we drop the notation ($\tilde{\cdot}$) throughout the paper.

\subsection{Uniformly boundedness of approximated moments }
Here, we prove the following lemma, which is required for the subsequent results.
\begin{lem}\label{Classicaluniboundlemma1}
 Let $T \in (0, \infty)$ and $g_0 \in  \Lambda^+_{-\sigma_1, \sigma_2}(0)$. Assume $g^n$ be the solution to (\ref{4trun cccecf})--(\ref{4trunc ccecfin1}). Then, for each $r \in \mathbb{R}$ with $ r \in [-\sigma_1, \sigma_2] $,
\begin{eqnarray}\label{momentfinite}
\mathcal{M}_{r, n}(g^n)(t):= \int_0^{\infty} z^r g^n(z, t)dz =\int_{0}^{n} z^r g^n(z, t)dz \leq \Gamma_r(T),
\end{eqnarray}
where $1 < \max \{ 1+\alpha, 1+\alpha' \} \le  \sigma_2 \le 2$ with $1/2 \le \max \{(1-\alpha), (1-\alpha') \} \le \sigma_1$, for $r=0$, $\Gamma_0(T)=\Gamma_0$ and $r=1$, $\Gamma_1(T)=\Gamma_1$.
\end{lem}

\begin{proof}

Let $r=1$. Then by the direct integration of (\ref{4trun cccecf}) by using (\ref{collkernel}) and (\ref{4trunc soln}) with respect to volume variable $z$ from $0$ to $n$ after multiplying with the weight $z$, we get
\begin{align}\label{4unibound1}
\frac{d}{dt} \mathcal{M}_{1, n}(g^n)&(t) =   \frac{1}{2} \int_0^{n} \int_{0}^{z} z E(z-z_1, z_1) G^n(z-z_1, z_1, t) dz_1 dz-\int_0^{n} \int_{0}^{n} z G^n(z, z_1, t) dz_1dz \nonumber\\
 & +\frac{1}{2}\int_0^{n} \int_{z}^{n}\int_{0}^{z_1} z P(z|z_1-z_2; z_2) E_1(z_1-z_2, z_2) G^n(z_1-z_2, z_2, t)dz_2dz_1 dz.
\end{align}
Applying Fubini's theorem to the first and the third integrals on the right-hand side to (\ref{4unibound1}) and using the transformation $z-z_1=z'$ $\&$ $z_1=z_1'$, we have
\begin{align}\label{4unibound2}
\frac{d}{dt} \mathcal{M}_{1, n}& (g^n)(t)=   \frac{1}{2} \int_0^{n} \int_{0}^{n-z} (z+z_1) E(z, z_1) G^n(z, z_1, t) dz_1 dz -\int_0^{n} \int_{0}^{n} z G^n(z, z_1, t)dz_1dz \nonumber\\
 & +\frac{1}{2}\int_0^{n} \int_{0}^{z_1}\int_{0}^{z_1}z P(z|z_1-z_2; z_2) E_1(z_1-z_2, z_2) G^n(z_1-z_2, z_2, t)dz dz_2 dz_1.
\end{align}

Using (\ref{mass1}), Fubini's theorem, replacing $z_1$ by $z$ $\&$ $z_2$ by $z_1$, the transformation $z-z_1=z'$ $\&$ $z_1=z_1'$ and the symmetry of $\Phi_n$ on the right-hand side to (\ref{4unibound2}), we obtain
\begin{align}\label{4unibound3}
\frac{d}{dt} \mathcal{M}_{1, n}(g^n)(t)=   &\frac{1}{2} \int_0^{n} \int_{0}^{n-z} (z+z_1) E(z, z_1) G^n(z, z_1, t) dz_1 dz -\int_0^{n} \int_{0}^{n} z G^n(z, z_1, t) dz_1dz \nonumber\\
 & +\frac{1}{2}\int_0^{n} \int_{0}^{n-z} (z+z_1) E_1(z, z_1) G^n(z, z_1, t)dz_1 dz\nonumber\\
 \leq & -\int_0^{n} \int_{n-z}^{n} z G^n(z, z_1, t)dz_1dz \leq 0.
\end{align}
 Integrating \eqref{4unibound3} with respect to time from $0$ to $t$ and using $g_0 \in \Lambda^{+}_{-\sigma_1, \sigma_2}(0)$, we obtain $\mathcal{M}_{1,n}(g^n)(t) \le \Gamma_1$, where $\Gamma_1$ is a constant .\\

 Next, we estimate the zeroth moment of $g^n$, by integrating (\ref{4trun cccecf}) with respect to $z$ from $0$ to $n$ and using (\ref{4trunc soln}), we have
\begin{align}\label{4unibound4}
\frac{d}{dt} \mathcal{M}_{0, n}(g^n)&(t) = \frac{1}{2} \int_0^{n} \int_{0}^{z}  E(z-z_1, z_1) G^n(z-z_1, z_1, t) dz_1 dz-\int_0^{n} \int_{0}^{n}  G^n(z, z_1, t)dz_1dz \nonumber\\
 & +\frac{1}{2}\int_0^{n} \int_{z}^{n}\int_{0}^{z_1} P(z|z_1-z_2; z_2) E_1(z_1-z_2, z_2) G^n(z_1-z_2, z_2, t)dz_2dz_1 dz.
\end{align}

Applying Fubini's theorem to the first and third integrals on the right-hand side to (\ref{4unibound4}), then using \eqref{binarybreakage} and the transformation $z-z_1=z'$ and $z_1=z_1'$, we have
\begin{align}\label{4unibound5}
\frac{d}{dt} \mathcal{M}_{0, n}(g^n)(t) \leq   &\frac{1}{2} \int_0^{n} \int_{0}^{n} E(z, z_1) G^n(z, z_1, t) dz_1 dz-\int_0^{n} \int_{0}^{n} G^n(z, z_1, t)dz_1dz \nonumber\\
 & + \int_0^{n} \int_{0}^{n} E_1(z, z_1)  G^n(z, z_1, t)dz_1 dz\nonumber\\
 \le   &-\frac{1}{2} \int_0^{n} \int_{0}^{n} E(z, z_1) G^n(z, z_1, t) dz_1 dz-\int_0^{n} \int_{0}^{n} E_1(z, z_1) G^n(z, z_1, t)dz_1dz \nonumber\\
 & + \int_0^{n} \int_{0}^{n} E_1(z, z_1)  G^n(z, z_1, t)dz_1 dz \le 0.
\end{align}

Then, using \eqref{4trunc ccecfin1} and $g_0 \in \Lambda^{+}_{-\sigma_1, \sigma_2}(0)$, we estimate
\begin{align*}
\mathcal{M}_{0, n}(g^n)(t) \le    \mathcal{M}_{0,n}(g^n)(0) =: \Gamma_0,
\end{align*}
where $\Gamma_0$ is a constant. Next, we first multiply $z^2$ to \eqref{4trun cccecf} and taking integration from $0$ to $n$ with respect to $z$. Then, using Fubini's theorem, the transformation $z-z_1=z'$ $\&$ $z_1=z_1'$ and Remark \ref{distribution alpha3} for $r=2$, we estimate
\begin{align}\label{4unibound7}
\frac{d}{dt} \mathcal{M}_{2,n}(g^n)(t)= &  \frac{1}{2} \int_0^{n} \int_{0}^{n-z} (z+z_1)^2 E(z, z_1) G^n(z, z_1, t) dz_1 dz\nonumber\\
&-\int_0^{n} \int_{0}^{n} z^2[E(z, z_1)+ E_1(z, z_1)] G^n(z, z_1, t)dz_1dz \nonumber\\
 & +\frac{1}{2} \times \frac{2}{3} \int_0^{n} \int_{0}^{n-z} ( z+z_1)^2  E_1(z, z_1) G^n(z, z_1, t) dz_1 dz\nonumber\\
 \leq &  \int_0^{n} \int_{0}^{n} zz_1 E(z, z_1) G^n(z, z_1, t) dz_1 dz\nonumber\\
 & + \int_0^{n} \int_{0}^{n}  zz_1   E_1(z, z_1) G^n(z, z_1, t)dz_1 dz= \int_0^{n} \int_{0}^{n}  zz_1  G^n(z, z_1, t)dz_1 dz.
\end{align}
 Then, applying $(\alpha_2)$ on the right-hand side to (\ref{4unibound7}) and estimates on $\mathcal{M}_{0, n}(g^n)(t)$ and $\mathcal{M}_{1, n}(g^n)(t)$, we evaluate
\begin{align}\label{4unibound8}
\frac{d}{dt} \mathcal{M}_{2,n}(g^n)(t)
 \le & c \bigg\{\int_0^{1} \int_{0}^{1} +\int_0^{1} \int_{1}^{n}+\int_1^{n} \int_{0}^{1}+\int_1^{n} \int_{1}^{n} \bigg\} z z_1  \nonumber\\
 &~~~~~~\times (z^{\alpha}z_1^{\alpha'}+z^{\alpha'}z_1^{\alpha})  g^n(z, t) g^n(z_1, t)dz_1 dz\nonumber\\
\le & 2c \Gamma_0^2+2c \Gamma_0 \int_{1}^{n}  z_1  ( z_1^{\alpha'}+ z_1^{\alpha}) g^n(z_1, t)dz_1 \nonumber\\
 &+ c \int_1^{n} \int_{1}^{n} z z_1  (z+z_1)  g^n(z, t) g^n(z_1, t)dz_1 dz \nonumber\\
 \le & 2c \Gamma_0^2+ 4c \Gamma_0 \mathcal{M}_{2,n}(g^n)(t) + 2c \Gamma_1 \mathcal{M}_{2,n}(g^n)(t) \nonumber\\
  = & 2c \Gamma_0^2+ 2c (2\Gamma_0  + \Gamma_1 ) \mathcal{M}_{2,n}(g^n)(t).
\end{align}

Then, using $g_0 \in \Lambda^{+}_{-\sigma_1, \sigma_2}(0)$ and applying Gronwall's inequality to (\ref{4unibound8}), we obtain
\begin{align*}
\mathcal{M}_{2,n}(g^n)(t) \leq \Gamma_2(T),
\end{align*}
where $\Gamma_2(T) := (\mathcal{M}_{2,n}(g)(0)+ 2c \Gamma_0^2)e^{2c (2\Gamma_0  + \Gamma_1 )T} $.\\

Now, we evaluate the uniform boundedness of $\mathcal{M}_{-\sigma_1, n}(t)$ for $\sigma_1 \in [1/2, 1)$. For this, we multiply \eqref{4trun cccecf} by $z^{-\sigma_1}$ and integrating with respect to $z$ from $0$ to $\infty$ and using Fubini's theorem, Remark \ref{distribution alpha3}, the transformation $z-z_1=z'$ $\&$ $z_1=z_1'$ and interchanging $z_2$ to $z_1$ $\&$ $z_1$ to $z$ to achieve
\begin{align}\label{unibound11}
\frac{d}{dt}\mathcal{M}_{-\sigma_1, n}(g^n)(t) \le &\frac{1}{2} \int_0^{n} \int_0^{n-z} [ (z+z_1)^{-\sigma_1}-z^{-\sigma_1}-z_1^{-\sigma_1} ]  G^n(z, z_1, t) dz_1 dz \nonumber\\
&- \int_0^{n} \int_{n-z}^{n}  z^{-\sigma_1}  G^n(z, z_1, t) dz_1 dz\nonumber\\
 &+ \frac{\zeta(\sigma_1)}{2} \int_0^{n} \int_0^{n-z} (z+z_1)^{-\sigma_1} G^n(z, z_1, t) dz_1 dz.
\end{align}
We know that $(z+z_1)^{-\omega} \leq z^{-\omega}$ for $\omega \ge 0$ which implies the non-positivity of the first term on the right-hand side to (\ref{unibound11}). The second is also non-positive. Therefore, we evaluate (\ref{unibound11}), by using $(\alpha_2)$ and $\max(1-\alpha, 1-\alpha' ) \le \sigma_1$, as
\begin{align}\label{unibound12}
\frac{d}{dt}\mathcal{M}_{-\sigma_1, n}(g^n)(t) \leq &  \frac{\zeta(\sigma_1)}{2} c \int_0^{n} \int_0^{n} (z+z_1)^{-\sigma_1}  (z^{\alpha}z_1^{\alpha'}+z^{\alpha'}z_1^{\alpha}) g^n(z, t) g^n(z_1, t) dz_1 dz\nonumber\\
& \le \frac{\zeta(\sigma_1)}{2} c \bigg\{  \int_0^{1} \int_0^{1}+\int_0^{1} \int_1^{n}+\int_1^{n} \int_0^{1}+\int_1^{n} \int_1^{n}   \bigg\} (z+z_1)^{-\sigma_1}\nonumber\\
&~~~~~~~~~~~~~~~~~~~~~~~~  (z^{\alpha}z_1^{\alpha'}+z^{\alpha'}z_1^{\alpha}) g^n(z, t) g^n(z_1, t) dz_1 dz\nonumber\\
\leq & \zeta(\sigma_1) c  \bigg[ \mathcal{M}_{-\sigma_1, n}(g^n)(t) \Gamma_0 +  2 \Gamma_0 \Gamma_1+ \Gamma_1^2 \bigg].
\end{align}
Applying Gronwall's inequality into (\ref{unibound12}), we obtain
\begin{align*}
\mathcal{M}_{-\sigma_1, n}(g^n)(t)\leq \Gamma_{-\sigma_1}(T),
\end{align*}
where $\Gamma_{-\sigma_1}(T):=  [\mathcal{M}_{-\sigma_1, n}(g)(0)+   \zeta(\sigma_1)c( 2\Gamma_0 \Gamma_1+ \Gamma_1^2)) e^{( \zeta(\sigma_1) c \Gamma_0  T  )} $. Thanks $g_0 \in \Lambda^{+}_{-\sigma_1, \sigma_2}(0)$. This completes the proof of Lemma \ref{Classicaluniboundlemma1}.
\end{proof}

In order to apply a compactness technique, we need the following lemma. In the next lemma, we show $\{g^n\}_{n \in \mathds{N}}$ is uniformly bounded on each compact set $\Xi (0, Z^o; T)$, for a fixed $T>0$ and $0 \le z \le Z^o$.
\subsection{Uniform boundedness of approximated solutions}
\begin{lem}\label{Uniform bounded of approximated solutions}
 Let $T \in (0, \infty)$ and $g_0 \in  \Lambda^+_{-\sigma_1, \sigma_2}(0)$. Assume $g^n$ be the solution to (\ref{4trun cccecf})--(\ref{4trunc ccecfin1}). Then, we have
\begin{eqnarray*}
g^n(z, t) \le \mathcal{S}(T) \ \ \text{on}\ \ \Xi (0, Z^o; T),
\end{eqnarray*}
 where $T>0$ and $0 \le z \le Z^o$.
\end{lem}

\begin{proof}
Since the second term on right-hand side of (\ref{4trun cccecf}) is non-positive, thus, from (\ref{4trun cccecf}), we have
\begin{align}\label{uniform1}
\frac {\partial g^n(z, t)}{\partial t} 
 \leq &\frac{1}{2}\int_{0}^{z} G^n(z-z_1, z_1, t)dz_1\nonumber\\   
  & +\frac{1}{2} \int_{z}^{n}\int_{0}^{z_1} P(z|z_1-z_2;z_2) G^n(z_1-z_2, z_2, t)dz_2 dz_1.   
\end{align}
Using $(\alpha_2)$ and $\eqref{binarybreakage}$ into (\ref{uniform1}), we estimate
\begin{align}\label{uniform2}
\frac {\partial g^n(z, t)} {\partial t} \leq &\frac{1}{2} c \int_{0}^{z} [ (z-z_1)^{\alpha}z_1^{\alpha'}+(z-z_1)^{\alpha'}z_1^{\alpha}]  g^n(z-z_1, t)g^n(z_1, t)dz_1 \nonumber\\
&+ c  \int_{z}^{n}\int_{0}^{z_1} \frac{1}{z_1} [ (z_1-z_2)^{\alpha}z_2^{\alpha'}+ (z_1-z_2)^{\alpha'}z_2^{\alpha} ]  g^n(z_1-z_2, t)g^n(z_2, t)dz_2 dz_1.
\end{align}
Applying Fubini's theorem to the second integral on the right-hand side to \eqref{uniform2} and using the transformation $z_1-z_2=z_1'$ $\&$ $z_2=z_2'$, we obtain
\begin{align}\label{uniform3}
\frac {\partial g^n(z, t)} {\partial t} \le & c (1+Z^o) \int_{0}^{z}   g^n(z-z_1, t)g^n(z_1, t)dz_1 \nonumber\\
&+  c   \int_{0}^{z}\int_{z-z_2}^{n-z_2} \frac{1}{(z_1+z_2)} ( z_1^{\alpha} z_2^{\alpha'}+ z_1^{\alpha'} z_2^{\alpha} ) g^n(z_1, t)g^n(z_2, t)dz_1 dz_2 \nonumber\\
&+ c   \int_{z}^{n}\int_{0}^{n-z_2} \frac{1}{(z_1+z_2)} ( z_1^{\alpha}z_2^{\alpha'}+ z_1^{\alpha'}z_2^{\alpha} )  g^n(z_1, t)g^n(z_2, t)dz_1 dz_2\nonumber\\
\le &c (1+Z^o) \int_{0}^{z}   g^n(z-z_1, t)g^n(z_1, t)dz_1 \nonumber\\
&+ c   \int_{0}^{n}\int_{0}^{n} ( z_1^{\alpha-1}z_2^{\alpha'}+ z_1^{\alpha'}z_2^{\alpha-1} )  g^n(z_1, t)g^n(z_2, t)dz_1 dz_2\nonumber\\
\le & c (1+Z^o) \int_{0}^{z}   g^n(z-z_1, t)g^n(z_1, t)dz_1 \nonumber\\
&+ 2c   \int_{0}^{n}\int_{0}^{n}  z_1^{\alpha-1}z_2^{\alpha'}   g^n(z_1, t)g^n(z_2, t)dz_1 dz_2.
\end{align}
An application of Lemma \ref{Classicaluniboundlemma1} and $1-\alpha \le \sigma_1$, we estimate \eqref{uniform3} as
\begin{align}\label{uniform4}
\frac {\partial g^n(z, t)} {\partial t} \leq  c (1+Z^o)  (g^n \ast g^n) (z, t) +   2c  ( \Gamma_0+ \Gamma_1) (\Gamma_{-\sigma_1}(T)+ \Gamma_1),
\end{align}
where  $(g^n \ast g^n)(z, t)$ is the convolution between $g^n$ with itself. Next, taking integration of \eqref{uniform4} with respect to time variable from $0$ to $t$, we have
\begin{align}\label{uniform5}
 g^n(z, t)  \le \tilde{g} (0)  +  \int_0^t  \bigg[ c (1+Z^o)    g^n \ast g^n(z, s) +  2c  ( \Gamma_0+ \Gamma_1) (\Gamma_{-\sigma_1}+ \Gamma_1) \bigg] ds,
\end{align}
where $\tilde{g}(0)=\sup_{0 \leq z \leq Z^o} g^n_0(z)$.
 Let us now define a function to control the right-hand side of the integral inequality (\ref{uniform5}) as
\begin{eqnarray}\label{uniform6}
A(z, t) :=  A(0) +  \int_0^t \bigg[ c (1+Z^o) (A \ast A)(z, s) + A(z, s) \bigg] ds, \ \ 0  \le z \le Z^o,\ 0 \le t \le T,
\end{eqnarray}

where $A(0)= \max\{ \tilde{g}(0),  2c  ( \Gamma_0+ \Gamma_1) (\Gamma_{-\sigma_1}(T)+ \Gamma_1) \}$ is a positive constant. Then, applying Laplace transform with respect to space variable $z$, using Leibniz's rule and then later its inverse Laplace transformation to (\ref{uniform6}), we have
\begin{eqnarray}\label{uniform7}
A(z, t) =  A(0) \exp \bigg(  c A(0) z (1+Z^o) (e^t-1)+t \bigg),\ \ 0 \le z \le Z^o,\ 0 \le t \leq T.
\end{eqnarray}
In order to complete the proof of the uniform boundedness of $g^n$, it is required to show that the following inequality holds
\begin{eqnarray}\label{uniform8}
g^n(z, t) \leq A(z, t),\ \ \ 0 \le z \le Z^o,\ 0 \leq t \leq T,\ n \in \mathds{N}.
\end{eqnarray}

(\ref{uniform8}) is shown by a contradiction. Next, we define the following auxiliary function as
\begin{eqnarray}\label{uniform9}
A_{\epsilon}(z, t):= A(0)+\epsilon + \int_0^t \bigg[ c (1+Z^o) (A_{\epsilon} * A_{\epsilon})(z, s) + A_{\epsilon}(z, s) \bigg] ds.
\end{eqnarray}
From (\ref{uniform5}) and (\ref{uniform6}), it is clear that $g^n_0(z) <  A_{\epsilon}(z, 0)$, for $0 \le z \le Z^o$. Let us assume that, for $n \geq 1$, there exists a set $Q$ such that $g^n(z, t) > A_{\epsilon}(z, t)$ for $(z, t)\in \Xi (0, Z^o; T)$. Since $Q$ does not contain points on the co-ordinate axes, we choose $(x_1, t_1)\in Q$ such that no points of $D:=[0, x_1)\times [0, t_1) $ in $Q$. Again, since $g^n$ and $A_{\epsilon}$ are continuous functions with respect to time variable and $g^n(z, t) \le A_{\epsilon}(z, t)$ in $D$, then we obtain
\begin{align}\label{uniform9}
g^n(x_1, t_1) > A_{\epsilon}(x_1, t_1)= & A(0)+\epsilon + \int_0^{t_1} \bigg[ c (1+Z^o) (A_{\epsilon} * A_{\epsilon})(x_1, s) + A_{\epsilon}(x_1, s) \bigg] ds\nonumber\\
> & g^n(x_1, 0)+\epsilon + \int_0^{t_1} \bigg[ c (1+Z^o) (g^n * g^n)(x_1, s) + g^n(x_1, s) \bigg] ds\nonumber\\
> & g^n(x_1, t_1),
\end{align}
which is a contradiction. This concludes that $Q$ is an empty set.
This gives
\begin{eqnarray}\label{uniform10}
g^n(z, t) \le A(z, t),\ \ \ 0 \leq z \leq Z^o ,\ 0 \leq t \leq T,\ n \in \mathds{N}.
\end{eqnarray}
Thus, we have
\begin{eqnarray*}
g^n(z, t) \leq \mathcal{S}(T),
\end{eqnarray*}
where $\mathcal{S}(T):=A(0) \exp \bigg( A(0) c (1+Z^o) Z^o (e^T-1)+T \bigg) $. Hence, the sequence $\{g^n\}_{n \in \mathds{N}}$ is uniformly bounded on $\Xi( 0, Z^o; T).$

\end{proof}

Further, we check the time  equicontinuity  of the family $\{g^n(t)\}$ on each compact set $\Xi (0, Z^o; T)$.
\subsection{Equicontinuity in time}
\begin{lem}\label{4equicontinuiuty}
Let $T \in (0, \infty)$ and $z \in [0, Z^{o}]$. Assume $g_0 \in \Lambda^+_{-\sigma_1, \sigma_2}(0)$. Suppose $g^n$ be a solution to \eqref{4trun cccecf}--\eqref{4trunc ccecfin1}. Then prove the following
\begin{eqnarray*}
\lim_{h \to 0} \sup_{t \in [0, T-h]} | g^n(z, t+h) - g^n (z, t)| =0.
\end{eqnarray*}
\end{lem}
\begin{proof}
Let $h \in (0, T)$ with $h<1$ and $t\in [0, T-h]$. Next, consider the following integral as
\begin{align}\label{4equicontinuiuty2}
|g^n(z, t+h) - & g^n (z, t) | \nonumber\\
\le & \frac{1}{2} \int_t^{t+h}\int_0^{z} E(z-z_1, z_1) G^n(z-z_1, z_1, s) dz_1 ds  +    \int_t^{t+h} \int_0^{n}  G^n(z, z_1, s)  dz_1 ds \nonumber\\
+ &  \int_t^{t+h}  \int_z^{n} \int_0^{z_1} \frac{1}{z_1} E_1(z_1-z_2, z_2) G^n(z_1-z_2, z_2, s)dz_2 dz_1 ds\nonumber\\
=:& G_1^n+G_2^n+G_3^n.
\end{align}
Let us estimate $G_1^n$, by applying $(\alpha_2)$ and Lemma \ref{Uniform bounded of approximated solutions}, as
\begin{align*}
G_1^n 
\leq &  \frac{c}{2} \int_t^{t+h} \int_0^z  [(z-z_1)^{\alpha}z_1^{\alpha'}+(z-z_1)^{\alpha'}z_1^{\alpha}] g^n(z-z_1, s) g^n(z_1, s)dz_1  ds \nonumber\\
\le &  \frac{c}{2}   \int_t^{t+h} \int_0^z  [(1+z)^{\alpha}(1+z_1)^{\alpha'}+(1+z)^{\alpha'}(1+z_1)^{\alpha}] g^n(z-z_1, s) g^n(z_1, s)dz_1  ds \nonumber\\
\le &  c (1+Z^o) \mathcal{S}(T) \int_t^{t+h} \int_0^z   g^n(z-z_1, s) dz_1  ds \le   c (1+Z^o) Z^o \mathcal{S}(T)^2 h.
\end{align*}
Similarly, $G_2^n$ can be estimated, by using $(\alpha_2)$, Lemma \ref{Classicaluniboundlemma1} and Lemma \ref{Uniform bounded of approximated solutions}, as
\begin{align*}
G_2^n = &  c \int_t^{t+h} \int_0^{n}  (z^{\alpha}z_1^{\alpha'}+z^{\alpha'}z_1^{\alpha}) g^n(z, s) g^n(z_1, s)dz_1  ds\nonumber\\
\leq  &  2c (1+Z^o)^{1/2} \mathcal{S}(T)  \int_t^{t+h} \int_0^{n} (1+z_1) g^n(z_1, s)dz_1  ds \le 2 c (1+Z^o)^{1/2} \mathcal{S}(T) (\Gamma_0 +\Gamma_1) h.
\end{align*}
An repeated application of Fubini's theorem, and using $(\alpha_2)$, \eqref{binarybreakage}, the transformation $z_1-z_2 =z_1'$ $\&$ $z_2 =z_2'$, Lemma \ref{Classicaluniboundlemma1} and $1-\alpha \le \sigma_1$, we evaluate $G_3^n$ as
\begin{align*} 
G_3^n \le &  c  \int_t^{t+h} \int_z^{n}   \int_0^{z_1} \frac{1}{z_1}
[(z_1-z_2)^{\alpha}z_2^{\alpha'}+(z_1-z_2)^{\alpha'}z_2^{\alpha}] g^n(z_1-z_2, s) g^n(z_2, s)dz_2 dz_1  ds\nonumber\\
  =&  c  \int_t^{t+h} \int_0^{z} \int_{z-z_2}^{n-z_2} \frac{1}{(z_1+z_2)}(z_1^{\alpha}z_2^{\alpha'}+z_1^{\alpha'}z_2^{\alpha}) g^n(z_1, s) g^n(z_2, s)dz_1 dz_2  ds\nonumber\\
 & + c   \int_t^{t+h} \int_z^{n} \int_0^{n-z_2} \frac{1}{(z_1+z_2)} (z_1^{\alpha}z_2^{\alpha'}+z_1^{\alpha'}z_2^{\alpha}) g^n(z_1, s) g^n(z_2, s)dz_1 dz_2  ds \nonumber\\
 \le &  c \int_t^{t+h} \int_0^{n} \int_{0}^{n}  (z_1^{\alpha-1}z_2^{\alpha'}+z_1^{\alpha'}z_2^{\alpha-1}) g^n(z_1, s) g^n(z_2, s)dz_1 dz_2  ds \nonumber\\
\le &  2c \int_t^{t+h} \int_0^{n} \int_{0}^{n}  z_1^{\alpha-1}(1+z_2)  g^n(z_1, s) g^n(z_2, s)dz_1 dz_2  ds \nonumber\\
\le &  2c (\Gamma_0+\Gamma_1) \int_t^{t+h} \bigg[ \int_{0}^{1}  z_1^{\alpha-1} g^n(z_1, s) dz_1+\int_{1}^{n}  z_1^{\alpha-1} g^n(z_1, s) dz_1 \bigg] ds\nonumber\\
\le &  2c (\Gamma_0+\Gamma_1)  ( \Gamma_{-\sigma_1}(T)+\Gamma_1)h.
\end{align*}
Inserting estimates on $G_1^n$, $G_2^n$ and $G_3^n$ into (\ref{4equicontinuiuty2}), we thus obtain
\begin{align*}
| g^n(z, t+h) - g^n (z, t) |
\le & c [   (1+Z^o) Z^o \mathcal{S}(T)^2 +  2 (1+Z^o)^{1/2} \mathcal{S}(T) ( \Gamma_0+\Gamma_1)\nonumber \\
& +   2(\Gamma_0+\Gamma_1)  (\Gamma_{-\sigma_1}(T)+\Gamma_1) ]h.
\end{align*}
As $h$ is an arbitrary. This completes the proof of lemma.
\end{proof}
Now, we state the following remark, which will help to prove next lemma and Theorem \ref{ClassicalExistence Thm}.
\begin{rmkk}\label{remark}
Let $\sigma(z)$ be a non-negative measurable function and $\lambda(z)$ is a positive and increasing function for $z>0$, then
\begin{align*}
\int_{\beta}^{\infty} \sigma(z) dz \le \frac{1}{\lambda(\beta)} \int_0^{\infty} \lambda(z) \sigma(z) dz, \ \text{for}\ \beta >0,
\end{align*}
if the integrals exist and are finite.
\end{rmkk}

Next, we show the equicontinuity of the family of solutions $\{g^n\}_{n\in\mathds{N}}$ with respect to the space variable $z$ in $\Xi (0, Z^{o}; T)$.
\subsection{Equicontinuity in space}
\begin{lem}\label{4equicontinuiutywithspace}
Assume $T \in (0, \infty)$ and $g_0 \in \Lambda^+_{-\sigma_1, \sigma_2}(0)$. Let $g^n$ be a solution to \eqref{4trun cccecf}--\eqref{4trunc ccecfin1}. Then, we have
\begin{eqnarray*}
\lim_{h \to 0} \sup_{t \in [0, T]} | g^n(z+h, t) - g^n (z, t)| =0,\ \ \text{in}\ \Xi (0, Z^{o}; T).
\end{eqnarray*}
\end{lem}
\begin{proof}
Let $0 \le z < z+h  \leq  Z^{o}$. Then for each $n\geq 1$, by applying triangle inequality and \eqref{binarybreakage}, we arrange the below term in the following manner
\begin{align}\label{eqconspace1}
|g^n&(z+h, t)-g^n(z, t)| \leq  |g_0^n(z+h )-g_0^n(z)|\nonumber\\
&+  \frac{1}{2} \bigg|\int_0^t \int_{z}^{z+h }E(z+h-z_1, z_1 ) G^n(z+h-z_1, z_1, s) dz_1 ds \bigg| \nonumber\\
&+ \frac{1}{2} \bigg| \int_0^t \int_0^{z} E(z+h-z_1, z_1) [\Phi_n(z+h-z_1, z_1, s)-  \Phi_n(z-z_1, z_1, s)]\nonumber\\
&~~~~~~~~~~~~~~~\times g^n(z-z_1, s) g^n(z_1, s)dz_1 ds \bigg| \nonumber\\
&+ \frac{1}{2} \bigg| \int_0^t \int_0^{z} E(z+h-z_1, z_1)  \Phi_n(z-z_1, z_1)[ g^n(z+h-z_1,s) -g^n(z-z_1,s)] g^n(z_1, s)dz_1 ds \bigg| \nonumber\\
&+ \frac{1}{2} \bigg| \int_0^t \int_0^{z} [E(z+h-z_1, z_1)-E(z-z_1, z_1)] G^n(z-z_1, z_1, s) dz_1 ds \bigg| \nonumber\\
&+\int_0^t | g^n(z+h, s)-g^n(z, s)|\int_0^{n}\Phi_n(z, z_1)g^n(z_1, s)dz_1 ds\nonumber\\
&+\int_0^t g^n(z+h, s)\int_0^{n} |\Phi_n(z+h, z_1)-\Phi_n(z, z_1)|g^n(z_1, s)dz_1 ds\nonumber\\
&+ \bigg| \int_0^t \int_z^{z+h} \int_0^{z_1} \frac{1}{z_1} E_1(z_1-z_2, z_2) G^n(z_1-z_2, z_2, s) dz_2 d z_1 ds\bigg|
=: \sum_{i=1}^{8} \mathcal{H}_i^n(t),
\end{align}
where $\mathcal{H}_i^n(t)$, for $i=1,2,3,\cdots,8$, are the corresponding integrands in the preceding line. In order to show $\{g^n\}_{n\in\mathds{N}}$ is equicontinuous with respect to $z$ it is sufficient to prove that as $h \to 0$, the left-hand side of \eqref{eqconspace1} approaches to zero. From \eqref{collkernel} and $(\alpha_1)$, we have $\Phi_n$ is uniformly continuous (not depending $n$ also) over $[0, Z^o] \times [0, Z_1^o]$. Thus, choose $h$ such that the following hold
\begin{eqnarray}\label{eqconspace2}
 \lim_{h \to 0} \sup_{z \in [0, Z^o] }  |g_0^n(z+h)-g_0^n(z)|=0,
\end{eqnarray}
\begin{eqnarray}\label{eqconspace3}
 \lim_{h \to 0} \sup_{(z, z_1) \in [0, Z^o] \times [0, Z_1^o] } |\Phi_n(z+h, z_1)-\Phi_n(z, z_1)|=0
 \end{eqnarray}
and
\begin{eqnarray}\label{eqconspaceE}
 \lim_{h \to 0} \sup_{(z, z_1) \in [0, Z^o] \times [0, Z_1^o] }  |E(z+h, z_1)-E(z, z_1)|=0,
 \end{eqnarray}
 where $0 \leq z < z+h \leq Z^o $, $0 \leq z_1  \leq Z_1^o$ with $Z^o>1 $ $\&$ $Z^o_1>1 $. The manner of choosing these $Z^o$ and $Z_1^o$ are described later. Next, we introduce the modulus of continuity as
 \begin{eqnarray*}
 \omega _n (t) := \sup_{t \in [0, T]} |g_n(z+h, t)-g_n(z, t)|,\ \ \ 0 \le z < z+h \leq Z^{o}.
 \end{eqnarray*}
Let us first estimate $\mathcal{H}_2^n(t)$, by using $(\alpha_2)$ and Lemma \ref{Uniform bounded of approximated solutions}, as
\begin{align*}
\mathcal{H}_2^n(t) 
\le & \frac{c}{2} \int_0^t \int_{z}^{z+h } [(z+h-z_1)^{\alpha}z_1^{\alpha'}+(z+h-z_1)^{\alpha'}z_1^{\alpha}]  g^n(z+h-z_1,s) g^n(z_1, s)dz_1 ds \nonumber\\
\leq & c (1+Z^o) \mathcal{S}(T)^2 T h.
\end{align*}
Using Lemma \ref{Uniform bounded of approximated solutions}, we evaluate $\mathcal{H}_3^n(t)$ as
\begin{align*}
\mathcal{H}_3^n(t) \le & \frac{1}{2} \int_0^t \int_0^{z} |\Phi_n(z+h-z_1, z_1)-  \Phi_n(z-z_1, z_1)| g^n(z-z_1, s)   g^n(z_1, s)dz_1 ds \nonumber\\
 \leq & \frac{1}{2} \sup_{(z, z_1) \in [0, Z^o] \times [0, z] }|\Phi_n(z+h-z_1, z_1)-\Phi_n(z-z_1, z_1)| \mathcal{S}(T)^2 Z^oT.
\end{align*}
Similarly, $\mathcal{H}_4^n(t)$ can be evaluated, by using $(\alpha_2)$, Lemma \ref{Uniform bounded of approximated solutions} and the definition of $\omega_n(s)$,  as
\begin{align*} 
\mathcal{H}_4^n(t) 
\le & \frac{1}{2}  \int_0^t \int_0^{z}  \Phi_n(z-z_1, z_1)| g^n(z+h-z_1,s) -g^n(z-z_1,s)| g^n(z_1, s)dz_1 ds \nonumber\\
\leq  & c (1+Z^o) \mathcal{S}(T) \int_0^t \int_0^{z}   | g^n(z+h-z_1,s) -g^n(z-z_1,s)| dz_1 ds \nonumber\\
\le & c (1+Z^o) \mathcal{S}(T) Z^o \int_0^t \omega_n(s)ds.
\end{align*}
Further, $\mathcal{H}_5^n(t)$ can be estimated, by applying $(\alpha_2)$ and Lemma \ref{Uniform bounded of approximated solutions}, as
\begin{align*}
\mathcal{H}_5^n(t) =& \frac{1}{2} \bigg| \int_0^t \int_0^{z} [E(z+h-z_1, z_1)-E(z-z_1, z_1)] \Phi_n(z-z_1, z_1) g^n(z-z_1,s)  g^n(z_1, s)dz_1 ds \bigg| \nonumber\\
\le &  2c (1+Z^o) Z^o  \mathcal{S}(T)^2T \sup_{(z, z_1) \in [0, Z^o] \times [0, z] } |E(z+h-z_1, z_1)-E(z-z_1, z_1)|.
\end{align*}
Next, $\mathcal{H}_6^n(t)$ can be evaluated, by using $(\alpha_2)$, Lemma \ref{Classicaluniboundlemma1} and the definition of $\omega_n(t)$, as
\begin{align*}
\mathcal{H}_6^n(t) 
\leq & c \int_0^t | g^n(z+h, s)-g^n(z, s)|\int_0^{n}    [z^{\alpha}z_1^{\alpha'}+z^{\alpha'}z_1^{\alpha}]    g^n(z_1, s)dz_1 ds\nonumber\\
\leq &   2c  (1+Z^o)^{1/2} (\Gamma_0+ \Gamma_1) \int_0^t \omega_n( s)  ds.
\end{align*}
By applying Lemma \ref{Uniform bounded of approximated solutions} and $(\alpha_2)$, we estimate $\mathcal{H}_7^n(t)$ as
\begin{align}\label{4H71}
\mathcal{H}_7^n(t)
\le & \int_0^t g^n(z+h, s)\int_0^{Z_1^o} |\Phi_n(z+h, z_1)-\Phi_n(z, z_1)|g^n(z_1, s)dz_1 ds\nonumber\\
&+\int_0^t g^n(z+h, s)\int_{Z_1^o}^{n} |\Phi_n(z+h, z_1)-\Phi_n(z, z_1)|g^n(z_1, s)dz_1 ds\nonumber\\
\le &  \mathcal{S}(T)   \bigg[ \mathcal{S}(T) \sup_{(z, z_1) \in [0, Z^o] \times [0, Z_1^o] }  |\Phi_n(z+h, z_1)-\Phi_n(z, z_1)| Z_1^o T \nonumber\\
 &~~~~~~~~~~+4c(1+Z^o+h) \int_0^t\int_{Z_1^o}^{\infty} z_1 g^n(z_1, s)dz_1 ds \bigg].
\end{align}
We choose $Z_1^o >1$ and $\epsilon(h) >0$ (depending on $h$) such that ${Z_1^o}^{-1}\Gamma_2(T) < \epsilon(h) $. Then, applying Remark \ref{remark} to (\ref{4H71}), we get
\begin{align*} 
\mathcal{H}_7^n(t) \le   \mathcal{S}(T)  [ \mathcal{S}(T) \sup_{(z, z_1) \in [0, Z^o] \times [0, Z_1^o] }  |\Phi_n(z+h, z_1)-\Phi_n(z, z_1)| Z_1^o T +4 c(1+Z^o+h)T\epsilon(h)].
\end{align*}
Finally, we evaluate $\mathcal{H}_8^n(t)$, by using \eqref{binarybreakage}, $(\alpha_2)$, Fubini's theorem, the transformation $z_1-z_2=z_1'$ $\&$ $z_2=z_2'$, Lemma \ref{Uniform bounded of approximated solutions} and Lemma \ref{Classicaluniboundlemma1}, as
\begin{align*}
\mathcal{H}_8^n(t) 
\le & c \int_0^t \int_z^{z+h} \int_0^{z_1}\frac{1}{z_1} [(z_1-z_2)^{\alpha}z_2^{\alpha'}+(z_1-z_2)^{\alpha'}z_2^{\alpha}]  g^n(z_1-z_2, s) g^n(z_2, s) dz_2 d z_1 ds  \nonumber\\
\le &  c \int_0^t \int_0^{z} \int_{z-z_2}^{z+h-z_2}\frac{1}{(z_1+z_2)} [ z_1^{\alpha}z_2^{\alpha'}+ z_1^{\alpha'}z_2^{\alpha}]  g^n(z_1, s) g^n(z_2, s) dz_1 d z_2 ds  \nonumber\\
 & +  c \int_0^t \int_z^{z+h} \int_0^{z+h-z_2}\frac{1}{(z_1+z_2)} [ z_1^{\alpha} z_2^{\alpha'}+ z_1^{\alpha'}z_2^{\alpha}]  g^n(z_1, s) g^n(z_2, s) dz_1 d z_2 ds  \nonumber\\
 \le &  c (1+Z^o)^{1/2} \mathcal{S}(T) \int_0^t \int_0^{z} \int_{z-z_2}^{z+h-z_2}  ( z_2^{\alpha-1}+ z_2^{\alpha'-1})  g^n(z_2, s)  dz_1 d z_2 ds  \nonumber\\
 & +  c (1+Z^o)^{1/2} \mathcal{S}(T) \int_0^t \int_z^{z+h} \int_0^{z+h-z_2}  (  z_1^{\alpha'-1}+ z_1^{\alpha-1})  g^n(z_1, s) dz_1 d z_2 ds  \nonumber\\
 \le & 4c  \mathcal{S}(T)  (\Gamma_{-\sigma_1}(T) +\Gamma_1) (1+Z^o)^{1/2} T h.
\end{align*}
Using \eqref{eqconspace2}, \eqref{eqconspace3} and \eqref{eqconspaceE},
 and then taking limit $h \to 0$ to the estimates on $\mathcal{H}_i^n(t)$, for $i=2,3,\cdots, 8$,  and finally inserting them into (\ref{eqconspace1}) for arbitrary $\epsilon(h) >0$, we have
\begin{align*}
\lim_{h \to 0}|g^n(z+h, t)-g^n(z, t)| \le [ c (1+Z^o) \mathcal{S}(T) Z^o +  2c  (1+Z^o)^{1/2} (\Gamma_0+ \Gamma_1) ]  \lim_{h \to 0} \int_0^t \omega_n( s)  ds.
\end{align*}
Then, by Gronwall's inequality, we obtain
\begin{align*}
\lim_{h \to 0} |g^n(z+h, t)-g_n(z, t)|  =0.
\end{align*}
This completes the proof of Lemma \ref{4equicontinuiutywithspace}.
\end{proof}
Next, from Lemma \ref{Uniform bounded of approximated solutions}, Lemma \ref{4equicontinuiuty}, Lemma \ref{4equicontinuiutywithspace} and Arzel\'{a}'s theorem \cite{Ash:1972, Edwards:1965}, we confirm that $\{g^n\}_{n \in \mathds{N}}$ is relatively compact in $\Xi(0, Z^o; T)$ which ensures that there exists a continuous function $g$ such that
\begin{equation}\label{limitconvergencestrong}
g^n \rightarrow g
\end{equation}
uniformly on each compact set $\Xi(0, Z^o; T)$ of $\Xi$.


\begin{proof}
\emph{of the Theorem \ref{ClassicalExistence Thm}}: In order to complete the proof of the Theorem \ref{ClassicalExistence Thm}, we require to use a diagonal method. According to this process, we select a subsequence of $\{g^n\}_{n \in \mathbb{N}}$ (say $\{g^j\}_{j=1}^{\infty}$) converging uniformly on each compact set in $\Xi$ to a non-negative continuous function $g$.\\
Let us consider the following integral as
\begin{align*}
\int_0^{u} (z^{-\sigma_1}+ z^{\sigma_2}) g(z, t)dz.
\end{align*}
Then for all $\epsilon >0$, there exists $j\geq 1$ such that
\begin{align*}
\int_0^{u} (z^{-\sigma_1}+ z^{\sigma_2}) |g(z, t)-g^j(z, t)|dz \leq \epsilon.
\end{align*}
Since $u$ and $\epsilon $ are arbitrary. Thus, from Lemma \ref{Classicaluniboundlemma1}, we obtain
\begin{align}\label{Existence0}
\int_0^{\infty} (z^{-\sigma_1}+ z^{\sigma_2})  g(z, t)dz \leq   \Gamma_{-\sigma_1}(T) +\Gamma_{\sigma_2}(T).
\end{align}
As we have $1/2 \le \sigma_1 < 1$ and $1< \sigma_2 \le 2$ and then from \eqref{Existence0}, we clear that $g \in \Lambda^+_{-\sigma_1, \sigma_2}(T) \cap ( \Lambda^+_{-1/2, -\sigma_1}(T) \cup  \Lambda^+_{1, \sigma_2}(T))$. Next, we require to show that $g$ is actually a solution to \eqref{cecbeclassi}--\eqref{4in1}. For this, let us consider the following, by using \eqref{binarybreakage}, as
\begin{align}\label{Existence1}
 (g^j-g)&(z, t)+g(z, t)\nonumber\\ =& g_0^j(z)+\int_0^t \bigg[ \frac{1}{2} \int_0^{z}E(z-z_1,z_1) (\Phi_j-\Phi)(z-z_1,z_1)g^j(z-z_1,s)g^j(z_1,s)dz_1 \nonumber\\
&+ \frac{1}{2} \int_0^{z} E(z-z_1,z_1) \Phi(z-z_1,z_1)[g^j(z-z_1,s)-g(z-z_1,s)]g^j(z_1,s)dz_1 \nonumber\\
&+ \frac{1}{2} \int_0^{z}  E(z-z_1,z_1) \Phi(z-z_1,z_1)[g^j(z_1,s)-g(z_1,s)] g(z-z_1,s)dz_1 + \mathcal{C}(g)(z, s) \nonumber\\
 &-g^j(z,s)\int_0^{\infty}(\Phi_j-\Phi)(z,z_1)   g^j(z_1,s)dz_1 -(g^j-g)(z,s)\int_0^{\infty}   \Phi(z, z_1) g^j(z_1,s)dz_1\nonumber \\
  &-g(z,s)\int_0^{\infty}   \Phi(z,z_1) (g^j-g)(z_1,s)dz_1 -\mathcal{B}(g)(z, s)+\mathcal{B}^*(g)(z, s) \nonumber \\
&+ \int_{z}^{\infty}\int_{0}^{z_1} \frac{1}{z_1}  E_1(z_1-z_2, z_2) (\Phi_j-\Phi)(z_1-z_2,z_2)  g^j(z_1-z_2,s)g^j(z_2,s)dz_2dz_1 \nonumber\\
&+\int_{z}^{\infty}\int_{0}^{z_1} \frac{1}{z_1}  E_1(z_1-z_2, z_2) \Phi(z_1-z_2,z_2) (g^j-g)(z_1-z_2,s)g(z_2,s)dz_2dz_1\nonumber\\
&+\int_{z}^{\infty}\int_{0}^{z_1} \frac{1}{z_1}  E_1(z_1-z_2, z_2) \Phi(z_1-z_2,z_2) (g^j-g)(z_2,s)g(z_1-z_2,s)dz_2dz_1\bigg]ds,
\end{align}
where we replace $\Phi_j$ and $g^j$ by $\Phi_{j}-\Phi+\Phi$ and $g^j-g+g$, respectively.\\

 Let us estimate the tail of the integral involved in the fifth term on the right-hand side to \eqref{Existence1}, by applying $(\alpha_2)$, Lemma \ref{Classicaluniboundlemma1} and Remark \ref{remark}, as
\begin{align}\label{Existence2}
\bigg|\int_{Z_1^o}^{\infty}(\Phi_j-\Phi)(z,z_1)  g^j(z_1,s)dz_1 \bigg| \leq 4c z^{1/2}  {Z_1^o}^{-1} {\Gamma}_2(T).
\end{align}
By applying $(\alpha_2)$ and Lemma \ref{Classicaluniboundlemma1}, we evaluate the sixth term on the right-hand side to \eqref{Existence1} as
\begin{align}\label{Existence3}
\bigg| \int_0^{\infty}  \Phi(z, z_1) g^j(z_1,s)dz_1 \bigg| \le 2c (1+z)^{1/2} ( {\Gamma}_0+ {\Gamma}_1).
\end{align}
Similarly, using $(\alpha_2)$, Remark \ref{remark} and Lemma \ref{Classicaluniboundlemma1}, we estimate tail of the seventh term on the right-hand side to \eqref{Existence1} as
\begin{align}\label{Existence4}
\bigg| \int_{Z_1^o}^{\infty}  \Phi(z,z_1) (g^j-g)(z_1,s)dz_1 \bigg| \leq 4c (1+z)^{1/2}  [ {Z_1^o}^{-1}{\Gamma}_2(T)  + {Z_1^o}^{-1} \|g\|_{ \max \{1+\alpha, 1+\alpha'\}}].
\end{align}
Let us consider the tenth term on the right-hand side to \eqref{Existence1}, by using Fubini's theorem, the transformation $z_1-z_2=z_1'$ $ \& $ $z_2=z_2'$ and replace $z_1\rightarrow z$ $ \& $ $z_2\rightarrow z_1$, as
\begin{align}\label{Eleven1}
\int_{z}^{\infty}\int_{0}^{z_1} & \frac{1}{z_1}   E_1(z_1-z_2, z_2) (\Phi_j-\Phi)(z_1-z_2,z_2)  g^j(z_1-z_2,s)g^j(z_2,s)dz_2dz_1 \nonumber\\
= & \int_{0}^{z}\int_{z-z_2}^{\infty} \frac{1}{(z_1+z_2)}  E_1(z_1, z_2) (\Phi_j-\Phi)(z_1,z_2)  g^j(z_1,s)g^j(z_2,s)dz_1dz_2 \nonumber\\
&+ \int_{z}^{\infty}\int_{0}^{\infty} \frac{1}{(z_1+z_2)}  E_1(z_1, z_2) (\Phi_j-\Phi)(z_1,z_2)  g^j(z_1,s)g^j(z_2,s)dz_1dz_2\nonumber\\
\le & \int_{0}^{\infty}\int_{0}^{\infty} \frac{1}{(z+z_1)}  E_1(z, z_1) |(\Phi_j-\Phi)(z, z_1)|  g^j(z, s)g^j(z_1, s)dz dz_1.
\end{align}
Now, we divide \eqref{Eleven1} into three sub-integrals as
\begin{align}\label{Eleven2}
 \int_{0}^{\infty}&\int_{0}^{\infty}  \frac{1}{(z+z_1)}  E_1(z, z_1) |(\Phi_j-\Phi)(z, z_1)|  g^j(z, s)g^j(z_1, s)dz dz_1 \nonumber\\
=& \bigg\{ \int_{0}^{Z^o}\int_{0}^{Z_1^o}+ \int_{Z^o}^{\infty}\int_{0}^{Z_1^o}+ \int_{0}^{\infty}\int_{Z_1^o}^{\infty} \bigg\}\frac{1}{(z+z_1)}  E_1(z, z_1)\nonumber\\
&~~~~~~~~~~~~~~\times |(\Phi_j-\Phi)(z, z_1)|  g^j(z, s)g^j(z_1, s)dz dz_1
\end{align}
We estimate the second and third integrals respectively, on the right-hand side to \eqref{Eleven2}, by using Lemma \ref{Classicaluniboundlemma1}, $(\alpha_2)$ and Remark \ref{remark}, as
\begin{align}\label{Eleven3}
\int_{Z^o}^{\infty}\int_{0}^{Z_1^o}& \frac{1}{(z+z_1)}  E_1(z, z_1) |(\Phi_j-\Phi)(z, z_1)|  g^j(z, s)g^j(z_1, s)dz dz_1 \nonumber\\
\le & 2c \int_{Z^o}^{\infty}\int_{0}^{Z_1^o} \frac{1}{(z+z_1)} (z^{\alpha} z_1^{\alpha'}+z^{\alpha'} z_1^{\alpha}) g^j(z, s)g^j(z_1, s)dz dz_1 \nonumber\\
\le & 4c (1+Z_1^o)^{1/2} \Gamma_0  {Z^o}^{-1}  \int_{Z^o}^{\infty} z_1 g^j(z_1, s) dz_1 \le 4c (1+Z_1^o)^{1/2} \Gamma_0  \Gamma_1 {Z^o}^{-1},
\end{align}
and
\begin{align}\label{Eleven4}
\int_0^{Z^o} \int_{Z_1^o}^{\infty} & \frac{1}{(z+z_1)}  E_1(z, z_1) |(\Phi_j-\Phi)(z, z_1)|  g^j(z, s)g^j(z_1, s)dz dz_1\nonumber\\
 \le & 4c (1+Z^o)^{1/2} \Gamma_0  \Gamma_1 {Z_1^o}^{-1}.
\end{align}
Next, consider the eleventh integral on the right-hand side to \eqref{Existence1}, by using Fubini's theorem, the transformation $z_1-z_2=z_1'$ $\&$ $z_2=z_2'$ and replace $z_1\rightarrow z$ $ \& $ $z_2\rightarrow z_1$, as
\begin{align}\label{Twelfth}
\int_{z}^{\infty}\int_{0}^{z_1} & \frac{1}{z_1}  E_1(z_1-z_2, z_2) \Phi(z_1-z_2,z_2) (g^j-g)(z_1-z_2,s)g(z_2,s)dz_2dz_1\nonumber\\
= & \int_{0}^{z}\int_{z-z_2}^{\infty} \frac{1}{(z_1+z_2)} E_1(z_1, z_2) \Phi(z_1, z_2) (g^j-g)(z_1, s)g(z_2, s) dz_1dz_2 \nonumber\\
&+ \int_{z}^{\infty}\int_{0}^{\infty} \frac{1}{(z_1+z_2)}  E_1(z_1, z_2) \Phi(z_1, z_2) (g^j-g)(z_1, s)g(z_2,s) dz_1dz_2\nonumber\\
\le & \int_{0}^{\infty}\int_{0}^{\infty} \frac{1}{(z+z_1)}  E_1(z, z_1) \Phi(z, z_1) |(g^j-g)(z, s)|g(z_1, s)dz dz_1.
\end{align}

We estimate the tail from \eqref{Twelfth}, by using Lemma \ref{Classicaluniboundlemma1}, $(\alpha_2)$ and Remark \ref{remark}, as
\begin{align}\label{Twelfth1}
 \int_{0}^{\infty}\int_{Z^o}^{\infty}& \frac{1}{(z+z_1)}  E_1(z, z_1) \Phi(z, z_1) |(g^j-g)(z, s)|g(z_1, s)dz dz_1 \nonumber\\
 \le & c\int_{0}^{\infty}\int_{Z^o}^{\infty} \frac{1}{(z+z_1)} (z^{\alpha} z_1^{\alpha'}+z^{\alpha'} z_1^{\alpha}) [(g^j+g)(z, s)] g(z_1, s)dz dz_1 \nonumber\\
 \le &  c \|g\|_{-\sigma_1, \sigma_2}  {Z^o}^{-1} \int_{Z^o}^{\infty}  (z^{1+\alpha}+z^{1+\alpha'}) [(g^j+g)(z, s)] dz\nonumber\\
  \le &  2c \|g\|_{-\sigma_1, \sigma_2}  {Z^o}^{-1} [\Gamma_2(T)+ \|g\|_{\sigma_2 } ].
\end{align}
Finally consider the last term on the right-hand side to \eqref{Existence1}, by using Fubini's theorem, the transformation $z_1-z_2=z_1'$ $\&$ $z_2=z_2'$ and replace $z_1\rightarrow z$ $ \& $ $z_2\rightarrow z_1$, as
\begin{align}\label{Lasttermclassical}
\int_{z}^{\infty}\int_{0}^{z_1} & \frac{1}{z_1}  E_1(z_1-z_2, z_2) \Phi(z_1-z_2,z_2) (g^j-g)(z_2,s)g(z_1-z_2,s)dz_2dz_1\nonumber\\
= & \int_{0}^{z}\int_{z-z_2}^{\infty} \frac{1}{(z_1+z_2)} E_1(z_1, z_2) \Phi(z_1, z_2) (g^j-g)(z_2, s)g(z_1, s) dz_1dz_2 \nonumber\\
&+ \int_{z}^{\infty}\int_{0}^{\infty} \frac{1}{(z_1+z_2)}  E_1(z_1, z_2) \Phi(z_1, z_2) (g^j-g)(z_2, s)g(z_1,s) dz_1dz_2\nonumber\\
\le & \int_{0}^{\infty}\int_{0}^{\infty} \frac{1}{(z+z_1)}  E_1(z, z_1) \Phi(z, z_1) |(g^j-g)(z_1, s)|g(z, s)dz dz_1.
\end{align}
By applying Lemma \ref{Classicaluniboundlemma1}, $(\alpha_2)$ and Remark \ref{remark}, we consider the following tail from \eqref{Lasttermclassical} as
\begin{align}\label{Lasttermclassical1}
\int_{Z_1^o}^{\infty}  \int_{0}^{\infty}& \frac{1}{(z+z_1)}  E_1(z, z_1) \Phi(z, z_1) |(g^j-g)(z_1, s)|g(z, s)dz dz_1 \nonumber\\
 \le & c\int_{Z_1^o}^{\infty} \int_{0}^{\infty} \frac{1}{(z+z_1)} (z^{\alpha} z_1^{\alpha'}+z^{\alpha'} z_1^{\alpha}) [(g^j+g)(z_1, s)] g(z, s)dz dz_1 \nonumber\\
 \le &  2c \|g\|_{-\sigma_1, \sigma_2} {Z_1^o}^{-1}  \int_{0}^{\infty}  z^{\sigma_2} [(g^j+g)(z_1, s)] dz_1\nonumber\\
  \le &  2c \|g\|_{-\sigma_1, \sigma_2} {Z_1^o}^{-1}   [\Gamma_2(T) +\|g\|_{\sigma_2} ].
\end{align}
Using a similar argument for choosing $Z^o$, $Z_1^o$ and $\epsilon(h) >0$ as discussed in Lemma \ref{4equicontinuiutywithspace}, we can easily show that the right-hand side of each integrals (\ref{Existence2}), (\ref{Existence4}), \eqref{Eleven3}, \eqref{Eleven4}, \eqref{Twelfth1} and \eqref{Lasttermclassical1} tend to zero as $\epsilon(h) \to 0$. Next, taking limit $j \to \infty$ in \eqref{Existence1}, it can easily be seen that all other difference terms tends to zero.\\

Finally, we obtain that the function $g$ is a solution to \eqref{cecbeclassi}--\eqref{4in1} written in the following integral form:
\begin{align}\label{Existence}
 g(z, t)=& g_0(z)+\int_0^t [\mathcal{C}(g)(z, t)-\mathcal{B}(g)(z, t)+\mathcal{B}^*(g)(z, t)]ds.
\end{align}
From above estimates and the continuity of $g$, we confirm that the right-hand to \eqref{cecbeclassi} is also continuous function on $\Xi$. Next,
taking differentiation of (\ref{Existence}) with respect to time variable $t$ and using Leibniz's rule, which ensures that $g$ is a continuous differentiable solution to \eqref{cecbeclassi}--\eqref{4in1} and from (\ref{Existence0}),  $g \in \Lambda^+_{-\sigma_1, \sigma_2}(T) \cap ( \Lambda^+_{-1/2, -\sigma_1}(T) \cup  \Lambda^+_{1, \sigma_2}(T))$.  This proves the existence of solutions of Theorem \ref{ClassicalExistence Thm}. In order to complete the proof of Theorem \ref{ClassicalExistence Thm}, we need to prove uniqueness of mass-conserving solution. It has been shown in the next section.
\end{proof}

\section{Uniqueness of classical solutions}

Let $g$ and $h$ in $\Lambda^+_{-\sigma_1, \sigma_2}(T) \cap ( \Lambda^+_{-1/2, -\sigma_1}(T) \cup  \Lambda^+_{1, \sigma_2}(T))$ be two solutions to \eqref{cecbeclassi}--\eqref{4in1} on $[0,T]$, where $T>0 $, with $ g(0)=h(0)$. Set $ H:=g-h$. We define $Q(t)$ as
\begin{eqnarray*}
Q(t):= \int_{0}^{\infty} (z^{-\sigma_1}+ z) |H(z,t)|dz.
\end{eqnarray*}
From the properties of the signum function, we get
\begin{eqnarray}\label{uni2 2}
Q(t) =\int_{0}^{\infty} (z^{-\sigma_1}+ z) \mbox{sgn}(H(z, t)) [g(z, t)-h(z, t)]dz,
\end{eqnarray}
where
\begin{align}\label{uni2 3}
 g(z, t)-h(z, t) =& \frac{1}{2}\int_{0}^{t}\int_{0}^{z}E(z-z_1, z_1)\Phi(z-z_1, z_1) [g(z-z_1,s)g(z_1,s)-h(z-z_1,s)h(z_1,s)]dz_1 ds\nonumber\\
 &- \int_{0}^{t}\int_{0}^{\infty} \Phi(z, z_1)[g(z, s)g(z_1, s)-h(z,s)h(z_1, s)]dz_1ds\nonumber\\
 &+ \frac{1}{2} \int_{0}^{t}\int_{z}^{\infty}\int_{0}^{z_1} \mbox{sgn(H(z, s))} P(z|z_1-z_2;z_2) E_1(z_1-z_2, z_2) \Phi(z_1-z_2, z_2)\nonumber\\
 &~~~~~~~~~~~\times [g(z_1-z_2,s)g(z_2,s)-h(z_1-z_2,s)h(z_2,s)] dz_2dz_1ds.
\end{align}
 Substituting (\ref{uni2 3}) into (\ref{uni2 2}), and using \eqref{binarybreakage} and repeated applications of Fubini's theorem, we simplify it further as
\begin{align}\label{uni2 4}
Q(t)=& \frac{1}{2} \int_{0}^{t} \int_{0}^{\infty} \int_{0}^{\infty} ((z+z_1)^{-\sigma_1}+ z+ z_1) \mbox{sgn}(H(z+z_1, s)) E(z, z_1)\Phi(z, z_1)\nonumber\\
 &~~~~~~~~~~~~~~~[g(z, s)g(z_1, s)-h(z, s)h(z_1, s)]dz dz_1 ds\nonumber\\
&- \int_{0}^{t} \int_{0}^{\infty} \int_{0}^{\infty} (z^{-\sigma_1}+ z) \mbox{sgn}(H(z, s))\Phi(z, z_1) [g(z, s)g(z_1, s)-h(z, s)h(z_1, s)]dz dz_1 ds\nonumber\\
 &+ \frac{1}{2} \int_{0}^{t} \int_0^{\infty} \int_{0}^{z_1}\int_{0}^{z_1} (z^{-\sigma_1}+ z) \mbox{sgn(H(z, s))} P(z|z_1-z_2;z_2) E_1(z_1-z_2, z_2) \Phi(z_1-z_2, z_2)\nonumber\\
 &~~~~~~~~~~~\times [g(z_1-z_2,s)g(z_2,s)-h(z_1-z_2,s)h(z_2,s)]dz dz_2dz_1ds\nonumber\\
  =& \frac{1}{2} \int_{0}^{t} \int_{0}^{\infty} \int_{0}^{\infty} [((z+z_1)^{-\sigma_1}+ z+ z_1) \mbox{sgn}(H(z+z_1, s)) - (z^{-\sigma_1}+ z) \mbox{sgn}(H(z, s))\nonumber\\& 
  - (z_1^{-\sigma_1}+ z_1)\mbox{sgn}(H(z_1, s)) ] E(z, z_1)    \Phi(z, z_1) [g(z, s)H(z_1, s)+h(z_1, s)H(z, s)]dz dz_1 ds\nonumber\\
 &+\frac{1}{2} \int_{0}^{t} \int_{0}^{\infty} \int_{0}^{\infty}  E_1(z_1, z_2) \bigg[ \frac{2}{(z_1+z_2)} \int_0^{z_1+z_2} (z^{-\sigma_1}+ z)\mbox{sgn}(H(z, s)) dz\nonumber\\ & -(z_1^{-\sigma_1}+z_1) \mbox{sgn}(H(z_1, s))
 -(z_2^{-\sigma_1}+z_2)\mbox{sgn}(H(z_2, s)) \bigg]  \Phi(z_1, z_2)\nonumber\\
 &~~~~~~~~\times [g(z_1, s)H(z_2, s)+h(z_2, s)H(z_1, s)]dz_2 dz_1 ds,
\end{align}
where
\begin{eqnarray*}\label{uni2 41}
g(z, s)g(z_1, s)-h(z, s)h(z_1, s)= g(z, s)H(z_1, s)+h(z_1, s)H(z, s).
\end{eqnarray*}

Now, let us define $A_1$ and $A_2$ by
\begin{align*}
A_1(z, z_1, t):=&((z+z_1)^{-\sigma_1}+ z+ z_1) \mbox{sgn}(H(z+z_1, t)) - (z^{-\sigma_1}+ z) \mbox{sgn}(H(z, t)) \\
 &- (z_1^{-\sigma_1}+ z_1)\mbox{sgn}(H(z_1, t))
\end{align*}
and
\begin{align*}
A_2(z_1, z_2, t):=&\frac{2}{(z_1+z_2)} \int_0^{z_1+z_2} (z^{-\sigma_1}+z) \mbox{sgn}(H(z, t)) dz -(z_1^{-\sigma_1}+z_1) \mbox{sgn}(H(z_1, t))\\
&-(z_2^{-\sigma_1}+z_2) \mbox{sgn}(H(z_2, t)).
\end{align*}
Substituting $A_1(z, z_1, t)$ and $A_2(z_1, z_2, t)$ into (\ref{uni2 4}) and then using (\ref{mass1}) $\&$ \eqref{binarybreakage}, we obtain
\begin{align}\label{sum}
Q(t)=& \frac{1}{2} \int_{0}^{t} \int_{0}^{\infty} \int_{0}^{\infty} A_1(z, z_1, s) E(z, z_1) \Phi(z, z_1)  g(z, s)H(z_1, s) dz dz_1 ds\nonumber\\
 &+ \frac{1}{2} \int_{0}^{t} \int_{0}^{\infty} \int_{0}^{\infty}  A_1(z, z_1, s) E(z, z_1) \Phi(z, z_1)  h(z_1, s)H(z, s) dz dz_1 ds\nonumber\\
 &+ \frac{1}{2} \int_{0}^{t} \int_0^{\infty} \int_{0}^{\infty} A_2(z_1, z_2, s) E_1(z_1, z_2)  \Phi(z_1, z_2)  g(z_1, s) H(z_2, s) dz_1dz_2ds\nonumber\\
 &+ \frac{1}{2} \int_{0}^{t} \int_0^{\infty} \int_{0}^{\infty} A_2(z_1, z_2, s) E_1(z_1, z_2) \Phi(z_1, z_2)h(z_2, s) H(z_1, s) dz_1dz_2ds\nonumber\\
 =:& \sum_{i=1}^{4} S_{i}(t),
\end{align}
where $S_{i}(t)$, for $i=1,2,3, 4,$ are the corresponding integrals in (\ref{sum}).\\

 Let us consider the following two estimates on $A_1$
\begin{align}\label{BoundA1}
A_1(z, z_1, t)&H(z_1, t)\nonumber\\=& [((z+z_1)^{-\sigma_1}+z+z_1) \mbox{sgn}(H(z+z_1, t)) -(z^{-\sigma_1}+z) \mbox{sgn}(H(z, t))\nonumber\\ &-(z_1^{-\sigma_1}+z_1)\mbox{sgn}(H(z_1, t))]H(z_1, t)\nonumber\\
\le & [z^{-\sigma_1}+z_1^{-\sigma_1}+z+z_1 +z^{-\sigma_1}+z-z_1^{-\sigma_1}-z_1] |H(z_1, t)| \nonumber\\
 \le & 2(z^{-\sigma_1}+z) |H(z_1, t)|.
\end{align}
and similarly, we have
\begin{align}\label{BoundA11}
A_1(z, z, t)H(z, t) \le 2(z_1^{-\sigma_1}+z_1) |H(z, t)|.
\end{align}
Next, we consider the estimates on $A_2$ as
\begin{align}\label{BoundA2}
A_2(z_1, z_2, t)H(z_2, t)=& \bigg[\frac{2}{(z_1+z_2)} \int_0^{z_1+z_2} (z^{-\sigma_1}+z) \mbox{sgn}(H(z, t)) dz -(z_1^{-\sigma_1}+z_1) \mbox{sgn}(H(z_1, t))\nonumber\\
&~~~~~~~~~~~~ -(z_2^{-\sigma_1}+z_2) \mbox{sgn}(H(z_2, t)) \bigg] H(z_2, t)\nonumber\\
\le & \bigg[\frac{2}{(z_1+z_2)}  \bigg[\frac{(z_1+z_2)^{1-\sigma_1}}{1-\sigma_1}+ \frac{(z_1+z_2)^2}{2} \bigg]+z_1^{-\sigma_1}+z_1\nonumber\\
&~~~~~~~~~~~~ -z_2^{-\sigma_1}-z_2 \bigg]  |H(z_2, t)|\nonumber\\
\le &   \bigg[\frac{2(z_1+z_2)^{-\sigma_1}}{1-\sigma_1}+ (z_1+z_2) +z_1^{-\sigma_1}+z_1 -z_2^{-\sigma_1}-z_2 \bigg]  |H(z_2, t)|\nonumber\\
\le &   \bigg[\frac{2 z_1^{-\sigma_1}}{1-\sigma_1}+ z_1 +z_1^{-\sigma_1}+z_1  \bigg]  |H(z_2, t)| \le \frac{3}{1-\sigma_1} (z_1^{-\sigma_1}+z_1)|H(z_2, t)|.
\end{align}
Similarly, we get
\begin{align}\label{BoundA21}
A_2(z_1, z_2, t)H(z_1, t) \le \frac{3}{1-\sigma_1} (z_2^{-\sigma_1}+z_2) |H(z_1, t)|.
\end{align}

Now, we evaluate $S_1(t)$, by using \eqref{BoundA1} and $(\alpha_2)$, as
\begin{align*}
S_1(t) \le & c\int_{0}^{t} \int_{0}^{\infty} \int_{0}^{\infty} (z^{-\sigma_1}+z) [z^{\alpha}z_1^{\alpha'}+ z^{\alpha'}z_1^{\alpha}] |H(z_1, s)|  g(z, s) dz dz_1 ds\nonumber\\
\le & c\int_{0}^{t} \int_{0}^{\infty}  \bigg[ \int_{0}^{1} (z^{-\sigma_1}+z) [z^{\alpha}z_1^{\alpha'}+ z^{\alpha'}z_1^{\alpha}] g(z, s) dz\\
 &+ \int_{1}^{\infty} (z^{-\sigma_1}+z) [z^{\alpha}z_1^{\alpha'}+ z^{\alpha'}z_1^{\alpha}] g(z, s) dz \bigg]  |H(z_1, s)|   dz_1 ds\nonumber\\
 \le & 2c \|g\|_{-\sigma_1, \sigma_2} \int_{0}^{t} \bigg[ \int_{0}^{1} (z_1^{\alpha'}+ z_1^{\alpha})   |H(z_1, s)|   dz_1+ \int_{1}^{\infty} (z_1^{\alpha'}+ z_1^{\alpha})   |H(z_1, s)|   dz_1 \bigg] ds \nonumber\\
  \le & 4c \|g\|_{-\sigma_1, \sigma_2} \int_{0}^{t} \int_{0}^{1} (z_1^{-\sigma_1}+z_1)  |H(z_1, s)|   dz_1 ds \le 4c \|g\|_{-\sigma_1, \sigma_2} \int_{0}^{t} Q(s)ds.
\end{align*}

 Similarly, by using \eqref{BoundA11}, $(\alpha_2)$, we estimate $S_2(t)$, as
\begin{align*}
S_2(t)\le  4c  \|h\|_{-\sigma_1, \sigma_2} \int_{0}^{t} Q(s)ds.
\end{align*}

Further, we estimate $S_3(t)$, by using $(\alpha_2)$ and \eqref{BoundA2}, as
\begin{align*}
S_3(t)\le & c \frac{3}{1-\sigma_1} \int_{0}^{t} \int_0^{\infty} \int_{0}^{\infty}(z_1^{-\sigma_1}+z_1)  [z_1^{\alpha}z_2^{\alpha'}+ z_1^{\alpha'}z_2^{\alpha}] |H(z_2, s)|    g(z_1, s)  dz_1dz_2ds\nonumber\\
\le & c \frac{3}{1-\sigma_1} \int_{0}^{t} \int_0^{\infty}  \bigg[ \int_{0}^{1}(z_1^{-\sigma_1}+z_1)  [z_1^{\alpha}z_2^{\alpha'}+ z_1^{\alpha'}z_2^{\alpha}]     g(z_1, s)  dz_1 \\
& +\int_{1}^{\infty}(z_1^{-\sigma_1}+z_1)  [z_1^{\alpha}z_2^{\alpha'}+ z_1^{\alpha'}z_2^{\alpha}]     g(z_1, s)  dz_1 \bigg] |H(z_2, s)| dz_2ds\nonumber\\
\le &   \frac{6c}{1-\sigma_1} \|g\|_{-\sigma_1, \sigma_2} \int_{0}^{t} \bigg[ \int_0^{1}  [z_2^{\alpha'}+ z_2^{\alpha}]   |H(z_2, s)| dz_2+\int_1^{\infty}  [z_2^{\alpha'}+ z_2^{\alpha}]   |H(z_2, s)| dz_2 \bigg] ds\nonumber\\
\le &   \frac{12c}{1-\sigma_1} \|g\|_{-\sigma_1, \sigma_2} \int_{0}^{t} Q(s) ds.
 \end{align*}

Finally, $S_4(t)$ can be evaluated, by using $(\alpha_2)$ and \eqref{BoundA21}, as
\begin{align*}
S_4(t)\le  \frac{12c}{1-\sigma_1} \|h\|_{-\sigma_1, \sigma_2} \int_{0}^{t} Q(s)  ds.
\end{align*}

Inserting the estimates on $S_1(t)$, $S_2(t)$, $S_3(t)$ and $S_4(t)$ into (\ref{sum}), we obtain
\begin{align*}
Q(t) \leq  \Theta \int_{0}^{t} Q(s) ds,
\end{align*}
where $\Theta :=  \bigg(4c+\frac{12c}{1-\sigma_1} \bigg)(\|g\|_{-\sigma_1, \sigma_2} +\|h\|_{-\sigma_1, \sigma_2} )$. Then by Gronwall's inequality, we have
\begin{align*}
Q(t) \leq  0 \times \exp(\Theta T) = 0.
\end{align*}
Therefore, $g(z, t)=h(z, t)$ a.e. for each $z \in [0, \infty)$ which confirms the uniqueness of solutions to \eqref{cecbeclassi}--\eqref{4in1}.\\

Finally, we prove $g$ in $\Lambda^+_{-\sigma_1, \sigma_2}(T) \cap ( \Lambda^+_{-1/2, -\sigma_1}(T) \cup  \Lambda^+_{1, \sigma_2}(T))$ is a mass conserving solution .\\

In order to show that $g$ is indeed a mass conserving solution to \eqref{cecbeclassi}--\eqref{4in1}, it is sufficient to show that $\mathcal{M}_1(t)=\mathcal{M}_1(0)$ for all $t \in (0, T]$. Multiplying \eqref{cecbeclassi} by $z$ and taking integration between $0$ to $\infty$ with respect to $z$, and  applying $(\alpha_2)$, \eqref{binarybreakage}, the norm of $g$ in $\Lambda_{-\sigma_1, \sigma_2}^{+}(T)$, one can easily be verify that
\begin{eqnarray*}
\frac{d \mathcal{M}_1(t)}{dt}=0\ \ \forall \ t \in (0, T].
\end{eqnarray*}
On integration yield with respect to time from $0$ to $t$, we have
\begin{eqnarray*}
\mathcal{M}_1(t) =\mathcal{M}_1(0)\ \ \forall \ t \in (0, T].
\end{eqnarray*}
This completes the proof of the Theorem \ref{ClassicalExistence Thm}.

\subsection*{Acknowledgment}
The first author, PKB would like to thank University Grant Commission (UGC), $6405/11/44$, India, for assisting Ph.D fellowship and the second author, AKG wish to thank Science and Engineering Research Board (SERB), Department of Science and Technology (DST), India for their funding support through the project $YSS/2015/001306$ for completing this work.

\end{document}